\documentclass[12pt]{amsart}
\usepackage{amssymb}
\usepackage{amsmath}
\usepackage{longtable}

\newcommand\skewperp{\angle}
\renewcommand{\t}{\mathfrak{t}}

\newcommand\im{\operatorname{im}}

\renewcommand\Re{\operatorname{Re}}
\renewcommand\Im{\operatorname{Im}}
\newcommand\Aut{\mathcal Aut}
\renewcommand\H{\mathbb{H}}
\newcommand\C{\mathbb C}
\newcommand\X{\mathfrak X}
\newcommand\D{\mathcal D}
\newcommand\R{\mathbb R}

\newcommand\Z{\mathbb Z}
\newcommand\Y{\mathcal Y}
\newcommand\XX{\mathbb X}
\newcommand\Oo{\mathcal O}

\newcommand{\Pic}{\operatorname{Pic}}

\newcommand\dquo{/\!/\!/}
\newtheorem{Thm}{Theorem}[section]
\newtheorem{Prop}[Thm]{Proposition}
\newtheorem{Cor}[Thm]{Corollary}
\newtheorem{Lem}[Thm]{Lemma}
\theoremstyle{definition}
\newtheorem{Ex}[Thm]{Example}

\newtheorem{defi}[Thm]{Definition}
\newtheorem{Rem}[Thm]{Remark}
\newtheorem{Prob}[Thm]{Problem}
\unitlength=1mm   \numberwithin{equation}{section}
\numberwithin{table}{section} \oddsidemargin=0cm
\author{Ivan V. Losev}
\title{Classification of multiplicity free
Hamiltonian actions of  complex tori on
Stein manifolds}
\thanks{{\it Key words and phrases}: complex tori, multiplicity free
Hamiltonian actions, Stein manifolds, moment maps} \thanks{{\it 2000 Mathematics Subject
Classification.}  53D20 (Primary), 32E10, 32M05 (Secondary)}
\begin{document}
\begin{abstract}
A Hamiltonian action of a complex torus on a symplectic complex
manifold is said to be {\it multiplicity free} if a general orbit
is a lagrangian submanifold. To any multiplicity free Hamiltonian
action of a complex torus  $T\cong (\C^\times)^n$ on a Stein
manifold $X$ we assign a certain 5-tuple consisting of a Stein
manifold $Y$, an \'{e}tale map $Y\rightarrow \t^*$, a set of
divisors on $Y$ and elements of $H^2(Y,\Z)^{\oplus n}, H^2(Y,\C)$.
We show that $X$ is uniquely determined by this invariants.
Furthermore, we describe all 5-tuples arising in this way.
\end{abstract}
\maketitle \markright{HAMILTONIAN ACTIONS OF COMPLEX
TORI}
\tableofcontents
\section{Introduction}
Let $X$ be a  smooth manifold with a symplectic form $\omega$ and
$T$ be a compact torus  acting on $X$ by symplectomorphisms. We
recall that the action $T:X$ is called {\it Hamiltonian} if there
are $n=\dim T$ functions $H_1,\ldots, H_n\in C^\infty(X)$ such
that
\begin{itemize}
\item[(H1)] $\{H_i,H_j\}=0$, where $\{\cdot,\cdot\}$ denotes the
Poisson  bracket on $X$ induced by $\omega$.
\item[(H2)] The skew-gradient of $H_i$ coincides with the velocity
vector field of $\xi_i$, where $\xi_1,\ldots,\xi_n$ is a basis of
the Lie algebra $\t$ of $T$.
\end{itemize}

It follows from (H1),(H2) that all functions $H_i$ are
$T$-invariant.

To a Hamiltonian action $T:X$ one assigns a map $\mu:X\rightarrow
\t^*, \langle\mu(x),\xi_i\rangle=H_i(x),$ called the {\it moment
map}. This map is $T$-equivariant and satisfies the following
identity $$\langle d_x\mu(v),\xi\rangle=\omega(\xi_*x,v), \forall
x\in X,v\in T_xX, \xi\in\t,$$ where $\xi_*x$ denotes the velocity
vector in $x$ induced by $\xi$. Note that the moment map is
defined uniquely up to the addition of a scalar. The manifold $X$
(or, more precisely, the quadruple $(X,\omega,T:X,\mu)$ is called
a {\it Hamiltonian $T$-manifold}.

We remark that, by (H1),(H2), any $T$-orbit in $X$ is an isotropic
submanifold. When a general orbit is lagrangian, the Hamiltonian
$T$-manifold $X$ is called {\it multiplicity free} or, shortly,
{\it MF} (one also uses the term {\it symplectic toric manifold}).

The most crucial result concerning {\it compact} MF  Hamiltonian
$T$-manifolds is their classification due to Delzant,
\cite{Delzant1}. If $T_0$ is the inefficiency kernel of the action
$T:X$ (that is, the set of all elements of $T$ acting trivially on
$X$), then the action $T/T_0:X$ is Hamiltonian with moment map
$p\circ\mu$. Here $\mu$ is the moment map for the action $T:X$ and
$p$ is the natural projection $\t^*\rightarrow (\t/\t_0)^*$. So it
is enough to deal with effective actions. By the famous result of
Atyah, \cite{Atyah}, and Guillemin-Sternberg, \cite{GS}, the image
$\mu(X)$ is a convex polytope in $\t^*$ (provided $X$ is compact).
This image is called the moment polytope of $X$. Delzant proved that
a compact MF Hamiltonian $T$-manifold is uniquely determined by its
moment polytope. He also found a necessary and sufficient condition
for a convex polytope in $\t^*$ to be a moment polytope of a compact
MF Hamiltonian manifold.

The goal of the present paper is to solve a similar classificational
problem for {\it complex} tori actions on complex symplectic
manifolds (with holomorphic symplectic form). The definition of a
(MF) Hamiltonian action is generalized to the complex case directly
(the functions $H_i$ are now taken from the algebra $\Oo(X)$ of
holomorphic functions). Let $T$ denote a complex torus
$(\C^\times)^n$ and $X$ a complex Hamiltonian $T$-manifold. Since
the moment map $\mu$ is holomorphic, the manifold $X$ is not compact
unless the action $T:X$ is trivial. In this paper we concentrate on
the case when $X$ is a Stein manifold.

The data  we use to classify  MF Hamiltonian Stein $T$-manifolds
are  more complicated than in the Delzant case. Below $X$ is a MF
Stein Hamiltonian $T$-manifold.

By the results of \cite{Snow}, there is a categorical quotient $Y$
for the action $T:X$, which is a Stein space (see Section
\ref{SECTION_prelim}  for details). Let $\pi:X\rightarrow Y$ be the
quotient map. It turns out that in our case $Y$ is even a Stein
manifold (Lemma \ref{Lem:1.1}).  Let $\psi$ be a unique holomorphic
map $Y\rightarrow \t^*$ such that $\mu=\psi\circ\pi$. One can show
that $\psi$ is \'{e}tale(=a local isomorphism), Lemma \ref{Lem:1.1}.
The manifold $Y$ and the map $\psi$ are the first two pieces of data
we need.

%Further, one can show (Lemma \ref{Lem:1.2}) that the stabilizer of
%any point of $X$ in $X$ is connected.
For a one-dimensional subgroup $T_0\subset T$ we denote by
$\widetilde{\D}(T_0)$ the set of all connected components $X_0$ of
the
fixed-point submanifold $X^{T_0}$ such that there is $x\in X_0$ with $T_x=T_0$. %We
%will see below (Lemma \ref{Lem:1.3}) that any element of
%$\widetilde{\D}(T_0)$ is
Set $\D(T_0):=\{\pi(X_0), X_0\in \widetilde{\D}(T_0)\}$. We will see
below (assertion 2 of Lemma \ref{Lem:1.1}) that any element of
$\D(T_0)$ is a connected component of
$\psi^{-1}(\alpha+\t_0^\perp)$, where $\alpha\in\t^*$ and
$\t_0^\perp$ denotes the annihilator of $\t_0$ in $\t^*$. Set
$\D:=\cup_{T_0}\D(T_0)$. The set of divisors $\D$ is the third piece
of data we need.

The fourth piece of data is a system of certain elements of
$H^2(Y,\Z)$.  To a character $\chi$ of $T$ we assign the coherent
sheaf $\Oo_{\chi}$ on $Y$ defined by $$\Oo_\chi(U)=\{f\in
\Oo(\pi^{-1}(U))| t.f=\chi(t)f, \forall t\in T\}.$$ We will see
(assertion 3 of Lemma \ref{Lem:1.1}) that $\Oo_\chi$ is a line
bundle for any $\chi$. Let us fix a basis $\chi_1,\ldots,\chi_n$ in
the character group $\X(T)$ of $T$. Since $Y$ is Stein, we see that
the groups $\Pic(Y)$ and $H^2(Y,\Z)$ are naturally isomorphic. So we
get $n$ cohomology classes $c_i:=[\Oo_{\chi_i}]\in H^2(Y,\Z)$. We
remark that $\chi\mapsto [\Oo_\chi]$ is not a group homomorphism
despite for any characters $\chi_1,\chi_2$ there is a natural map
$\Oo_{\chi_1}\otimes\Oo_{\chi_2}\rightarrow \Oo_{\chi_1+\chi_2}$.

Finally, we will see (assertion 4 of Lemma \ref{Lem:1.1}) that
there is  a holomorphic 2-form $\omega_0$ on $Y$ such that
$\omega=\pi^*(\omega_0)-d\alpha$, where $\alpha$ is a
$T$-invariant holomorphic 1-form on $X$ satisfying
$\alpha(\xi_*)=H_\xi$ for any $\xi\in \t$; here and below
$H_\xi=\langle\mu,\xi\rangle$. Moreover, the class $c_0$ of
$\omega_0$ in the second DeRham cohomology group $H^2_{DR}(Y)$ is
well-defined. Since $Y$ is Stein, we know that $H^2_{DR}(Y)$ is
naturally isomorphic to $H^2(Y,\C)$, see \cite{Onishchik}, Theorem
4.16.

So to any  MF Hamiltonian Stein $T$-manifold $X$ we have assigned
the 5-tuple $\Y_X:=(Y,\psi,\D,(c_i)_{i=1}^n,c_0)$. Now let
$\Y:=(Y,\psi,\D, (c_i)_{i=1}^n,c_0),\Y':=(Y',\psi',\D',
(c_i')_{i=1}^n,c_0')$ be two 5-tuples of the indicated form. By a
morphism between $\Y,\Y'$ we mean an \'{e}tale holomorphic map
$\varphi:Y\rightarrow Y'$ satisfying the following conditions:
\begin{itemize}
\item[(a)] $\psi=\psi'\circ\varphi$.
\item[(b)] $\D$ is the set of connected components of $\bigcup_{Y_0'\in
\D'}\varphi^{-1}(Y_0)$.
\item[(c)] $c_i:=\varphi^*(c_i'), i=\overline{0,n}$.
\end{itemize}

Let $X,X'$ be MF Hamiltonian Stein $T$-manifolds (recall that $T$
is assumed to act effectively on both $X,X'$) and
$\omega,\omega',\mu,\mu'$ be the corresponding symplectic forms
and the moment maps. By a {\it Hamiltonian morphism} between $X$
and $X'$ we mean a holomorphic $T$-equivariant map $f:X\rightarrow
X'$ such that $f^*(\omega')=\omega, \mu=\mu'\circ f$. Since $\psi,\psi'$ are
both \'{e}tale, we easily see that a Hamiltonian morphism
$f:X\rightarrow X'$ gives rise to a unique morphism
$\Y_f:\Y_{X}\rightarrow \Y_{X'}$ such that $\Y_f\circ\pi=\pi'\circ
\Y_f$, where $\pi,\pi'$ denote the quotient maps for $X,X'$.

\begin{Thm}\label{Thm:1}
Let $X,X'$ be such as above. For any morphism
$\varphi:\Y_X\rightarrow \Y_{X'}$ there is a Hamiltonian morphism
$f:X\rightarrow X'$ with $\Y_f=\varphi$. In particular, $X,X'$ are
isomorphic if  $\Y_X,\Y_{X'}$ are.
\end{Thm}

We note that such a morphism $f$ is not necessarily unique. For
example, shifting $f$ by an element of $T$, we again get a
Hamiltonian morphism with desired properties.

Our next task is to give a characterization of  5-tuples $\Y$ of the
form $\Y_X$. We say that a 5-tuple $\Y$ is {\it Delzant} if it
enjoys the following property:
\begin{itemize}
\item[(Del)] Let $y$ be an arbitrary point of $Y$, $Y_0^1,\ldots, Y_0^k$
be all elements of $\D$ containing $y$ and $T_0^1,\ldots, T_0^k$ be
the corresponding connected one-dimensional subgroups of $T$. Then
there is a neighborhood $U$ of $y$ in $Y$ such that any element of
$\D$ intersecting $U$ is one of $Y_0^i,i=\overline{1,k}$.
Furthermore, the primitive elements of $\X(T_0^i)^*\hookrightarrow
\X(T)^*, i=\overline{1,k},$ (each of them is defined uniquely up to
the multiplication by $\pm 1$) constitute a part of a basis in the
lattice $\X(T)^*$.
\end{itemize}

\begin{Thm}\label{Thm:2}
Suppose a 5-tuple $\Y$ is  Delzant. Then there is a MF Hamiltonian
Stein $T$-manifold $X$ such that $\Y=\Y_X$.
\end{Thm}

Let us describe the structure of this paper. Section 2 is devoted to
some general results concerning MF Hamiltonian Stein $T$-manifolds.
Firstly, we recall results of Snow concerning general reductive
group actions on Stein manifolds. Then we establish the slice
theorem \ref{Thm:1.1}, which asserts that locally a MF Hamiltonian
Stein manifold looks like a model manifold introduced in Example
\ref{Ex:1}. Using this theorem we prove some structure results on MF
Hamiltonian Stein manifolds, Lemmas \ref{Lem:1.2},\ref{Lem:1.1}. In
Section 3 we establish a certain sheaf of groups $\Aut$  on the
quotient $Y$ of a MF Hamiltonian Stein manifold $X$. Its sections
are Hamiltonian automorphisms that preserve the quotient map . This
sheaf is important because it controls gluing of local pieces of an
MF Hamiltonian manifolds together. The main goal of the Section is
to determine the sheaf $\Aut$ up to an isomorphism (Lemmas
\ref{Lem:2.2},\ref{Lem:2.3}). In Section 4 we prove Theorem
\ref{Thm:1}. The most crucial part of its proof is the special case
of the identity morphism $\varphi$. To examine this case we consider
a natural action of $H^1(Y,\Aut)$ on the set of all isomorphism
classes of MF Hamiltonian Stein $T$-manifolds $X$ such that
$\Y_X=(Y,\psi,\D,\bullet,\bullet)$. Finally, in Section 5 we prove
Theorem \ref{Thm:2}. The main idea of the proof is to consider
certain local lagrangian sections of the quotient map.

{\bf Acknowledgements.} Part of the work on this problem was done
during author's stay in Rutgers University, New Brunswick,  in the
beginning of 2007. I would like to thank this institution and
especially  Professor F. Knop for hospitality. I also express my
gratitude to F. Knop for stimulating discussions. Finally, I would
like to thank J. Martens for useful comments on an earlier version
of this article.

\section{Generalities on MF Hamiltonian Stein $T$-manifolds}\label{SECTION_prelim}
Until a further notice $T$ denotes a complex torus $(\C^\times)^n$
and $X$ is a Stein $T$-manifold.

By \cite{Snow}, there is a Stein space $Y$ and a surjective
$T$-invariant holomorphic map $\pi:X\rightarrow Y$ satisfying the
following conditions:
\begin{itemize} \item[(A)] $U\subset Y$ is open iff
$\pi^{-1}(U)$ is,\item[(B)] $\pi^*(\Oo(U))=\Oo(\pi^{-1}(U))^T$ for
any open subset $U\subset Y$.\end{itemize} Such  $Y,\pi$ have the
following universal property: for any Stein space  $Z$ and a
$T$-equivariant holomorphic map $\rho:X\rightarrow Z$ there is a
unique $T$-equivariant holomorphic map
$\underline{\rho}:Y\rightarrow Z$ such that
$\rho=\underline{\rho}\circ\pi$. Therefore $Y$ is called the {\it
categorical quotient} for the action $T:X$ and $\pi$ is called the
quotient map. It is known, see \cite{Snow}, that any fiber of $\pi$
contains a unique closed $T$-orbit.

In the sequel we will need the following notion.
\begin{defi}\label{defi:2}
A subset $X_0\subset X$ is called {\it saturated} if it  consists
of fibers of the quotient map, that is $X_0=\pi^{-1}(\pi(X_0))$.
\end{defi}

If $X_0$ is open and saturated, then $\pi(X_0)$ is open.

Now we quote the Snow slice theorem. Let $x\in X$ be a point with
closed $T$-orbit. Set $T_0:=T_x$ and let $V$ denote the slice
$T_0$-module $T_xX/\t_*x$. Then there is an open  saturated
neighborhood $U$ of $x$ in $X$ and an open $T_0$-stable neighborhood
$U_0$ of $0$ in $V$ such that $U$ is $T$-equivariantly isomorphic to
the homogeneous bundle $T*_{T_0}U_0:=(T\times U_0)/T_0$.  Let
$\pi_0$ denote the quotient map for the action $T_0:V$. From
property (B) above it follows that $\pi(U)\cong\pi_0(U_0)$ and the
quotient $\pi:U\rightarrow \pi(U)$ is identified with
$T*_{T_0}U_0\rightarrow \pi_0(U_0)$. Since $U$ is saturated, we see
that $U_0$ is also saturated (w.r.t. $T_0$).

Now we would like to produce some examples of MF Hamiltonian Stein
$T$-manifolds.

\begin{Ex}\label{Ex:1}
Fix a connected subgroup $T_0\subset T$, characters
$\chi_1,\ldots,\chi_k$ of $T_0$ forming a basis of $\X(T_0)$, and
$\lambda\in \t^*$. To this data we assign a MF Hamiltonian Stein
manifold $X$ as follows. Let $V$ be a vector space with a basis
$v_1,\ldots,v_k$. Define the linear action $T_0:V$ by
$t.v_i=\chi_i(t)v_i, i=\overline{1,k}, t\in T_0$. Let $X_0$ be the
homogeneous vector bundle $T*_{T_0}V$. There is a subtorus
$T_1\subset T$ such that $T_1\times T_0=T$. For any such $T_1$ we
have $X_0=T_1\times V$. Choose a basis
$\theta_1,\ldots,\theta_l,l=n-k,$ in $\X(T_1)$, set
$\beta_i:=\frac{d\theta_i}{\theta_i}=d(\ln\theta_i)$. One can regard
$\beta_1,\ldots,\beta_l$ as a basis in $\t_1^*$. Let
$\beta^i,i=\overline{1,l},$ be the basis in $\t_1$ dual to
$\beta_i$. Further, let $v^1,\ldots,v^k$ be the basis of $V^*$ dual
to $v_1,\ldots,v_k$. Set $X:=T^*X_0$. We can consider
$\theta_i,\beta^i,v_j,v^j$ as functions on $X$ and these functions
form a coordinate system. The manifold $X$ has the natural
symplectic form $\omega=-d\alpha$, where
$\alpha:=\sum_{i=1}^l\beta^id\ln\theta_i+\sum_{i=1}^k v_i dv^i$. The
form $\alpha$ is $T$-invariant and the action $T:X$ is Hamiltonian
with moment map $\mu(x)=\lambda+\sum_{i=1}^l
\beta^i(x)\beta_i+\sum_{i=1}^k v_i(x)v^i(x) d(\ln\chi_i)$. One can
easily check that $\alpha,\mu$ do not depend on choices we have made
(i.e., the choices of $v_i,T_1,\theta_j$). We also note that the
action $T:X$ is effective and $\alpha(\xi_*)=H_\xi$.

Let us check that the 5-tuple $\Y_X$ is well-defined and compute it.
We have $Y\cong \C^n$, $\beta^1,\ldots,\beta^l, v_1v^1,\ldots,
v_kv^k$ being coordinates. The map $\psi$ is identical. We easily
get $\D=\{D_i,i=\overline{1,k}\}$, where $\D_i$ is the affine
hyperplane of the form $v_iv^i=const$ containing $\lambda$.

Choose $\chi\in \X(T), \chi=\sum_{i=1}^l a_i\theta_i+\sum_{i=1}^k
b_i\chi_i$. Then $\Oo_\chi$ is a trivial line bundle on $Y$
generated by $(\prod_{i=1}^l\theta_i^{a_i})(\prod_{i=1}^k
u_i^{|b_i|})$, where $u_i=v_i$ if $b_i\geqslant 0$ and $v^i$
otherwise. So $c_1,\ldots,c_n$ are well-defined and equal to zero
for any choice of a basis in $\X(T)$.

Since $\omega$ is exact and $H^2(Y,\C)=0$, we see that $c_0$ in this
case is well-defined and equals $0$.

We remark that the 5-tuple $\Y_X$ is Delzant.
%The manifold $X:=T^*(X_0)$ has the natural symplectic form
%$\omega=-d\alpha$, where $\alpha$ is the canonical 1-form defined by
%$\langle\alpha,\xi\rangle_{(x_0,\beta)}=\langle\beta,d_{(x_0,\beta)}\pi(\xi)\rangle.$
%Here $x_0\in X_0, \beta\in T^*_{x_0}X_0, \xi\in T_{(x_0,\beta)}X$
%and $\pi$ is the natural projection $X\twoheadrightarrow X_0$.

We call $X$  a {\it model Hamiltonian $T$-manifold} associated
with $T_0,\lambda,\chi_1,\ldots,\chi_k$. By a {\it base point} in $X$
we mean any point from a unique closed $T$-orbit in $X_0$ (i.e.,
any point $x\in X_0$ with $v^1(x)=\ldots=v^k(x)=0$).
\end{Ex}

\begin{defi}\label{defi:1.1}
Let $X_1,X_2$ be MF Hamiltonian Stein $T$-manifolds, $x_1\in
X_1,x_2\in X_2$ be points with closed $T$-orbits. We say that the
pairs $(X_1,x_1), (X_2,x_2)$ are {\it locally equivalent} if there
are saturated open neighborhoods $U_i,i=1,2,$ of $x_i$ in $X_i$ and
a Hamiltonian isomorphism  $\varphi:U_1\rightarrow U_2$ such that
$\varphi(x_1)=x_2$.
\end{defi}

Till the end of the section $X$ is an arbitrary MF Hamiltonian
Stein $T$-manifold.

\begin{Thm}[symplectic slice]\label{Thm:1.1}
Let $x\in X$ be a point with closed orbit, $T_0=T_x, \lambda=\mu(x)$.
Then the following assertions hold.
\begin{enumerate}
\item $T_0$ is connected, and the set of weights of $T_0$ in $T_xX$
has the form $\{0,\chi_1,\ldots,\chi_k,-\chi_1,$
$\ldots,-\chi_k\}$, where $\chi_1,\ldots,\chi_k$ is a basis of the
character group $\X(T_0)$.
\item Let $X'$ be a model Hamiltonian $T$-manifold associated with
$T_0,\lambda,\chi_1,\ldots,\chi_k$ and $x'$ its base point. Then the
pairs $(X,x)$ and $(X',x')$ are locally equivalent.
\end{enumerate}
\end{Thm}
\begin{proof}
Note that $T_xX$ is a symplectic $T_0$-module. There is a trivial
isotropic $T_0$-submodule $\t_*x\subset T_xX$. Therefore
$T_xX/(\t_*x)^\skewperp$, where $^\skewperp$ denotes the
skew-orthogonal complement, is also a trivial $T_0$-module. Set
$V:=(\t_*x)^\skewperp/\t_*x$. This is a symplectic $T_0$-module of
dimension $\dim X-2\dim Tx=2\dim T_0$. By Snow's slice theorem, the
action $T_0:V$ is effective. It follows that the weights of
$T_0^\circ$ in $V$ have the form
$\chi_1,\ldots,\chi_k,-\chi_1,\ldots,-\chi_k$, where
$\chi_1,\ldots,\chi_k$ form a basis of $\X(T_0^\circ)$. Therefore
$T_0^\circ$ is a maximal torus in $\operatorname{Sp}(V)$ whence
$T_0=T_0^\circ$. So assertion 1 is proved.

In view of the Snow slice theorem, the proof of assertion 2 is
completely analogous to that of the symplectic slice theorem from
\cite{slice} (where it was proved for Hamiltonian reductive group
actions on affine algebraic varieties).
\end{proof}

\begin{Cor}\label{Cor:1.1}
Let $y\in Y$ and $U$ be an open neighborhood of $y$ in $Y$. Then
there is an open Stein neighborhood $U_0$ of $y$ in $U$ such that
$\pi^{-1}(U_0)$ is Stein.
\end{Cor}
\begin{proof}
Thanks to Theorem \ref{Thm:1.1}, it is enough to consider the case
when $X$ is model and $y$ is the image of its base point $x$. Since the
product of two Stein manifolds is again Stein, the
claim can be easily reduced to the case when (in the notation of
Example \ref{Ex:1}) $n=1$. If $k=0$, then
$\pi^{-1}(U_0)\cong U_0\times\C^\times$. Let $k=1$. Here we need
to check that the domain in $\C^2$ given by $|v_1v^1|<d$ is Stein
for all (equivalently, some) positive $d$. But the latter is a
holomorphic region whence a Stein manifold (see \cite{Onishchik},
Example 4.1).
\end{proof}

\begin{Rem}\label{Rem:1.1}
Now let $X^1,X^2$ be model Hamiltonian $T$-manifolds corresponding
to the data $(T_0^1,\lambda^1,\chi^1_1,\ldots,\chi^1_{k_1}),
(T_0^2,\lambda^2,\chi_1^2,\ldots,\chi_{k_2}^2)$ and $x^1,x^2$ be the
corresponding base points. Then the pairs $(X^1,x^1),(X^2,x^2)$ are
locally equivalent iff the following conditions are satisfied:
\begin{enumerate}
\item $T_0^1=T_0^2$ whence $k_1=k_2$.
\item $\lambda^1=\lambda^2$.
\item After reordering $\chi^2_1,\ldots,\chi^2_{k_2}$, we get $\chi_i^1=\pm
\chi^2_i, i=\overline{1,k_1}$.
\end{enumerate}
\end{Rem}

Below we will often use a certain open subvariety of
$X$ associated with a "general enough" element $\zeta\in
\X(T)\otimes \R$.

We say that  $\zeta\in \X(T)\otimes \R$ is {\it general} for $X$
if $\zeta$ is nonzero on the Lie algebra of $T_x$ for any $x\in
X$. Since the set of all $T_x$ is countable, we see that  such
$\zeta$ does exist. By $X^\zeta$ we denote the subset of $X$
consisting of all points $x$ satisfying the following two
conditions:
\begin{enumerate}
\item $T_x=\{1\}$.
\item Let $\tau:\C^\times\rightarrow T$ be a one-parameter subgroup.
If $\lim_{t\rightarrow 0}\tau(t).x$ exists in $X$, then
$\langle\zeta,\tau\rangle>0$ (we note that the l.h.s of the
previous inequality is always nonzero because $\im\tau$ lies in
the stabilizer of $\lim_{t\rightarrow 0}\tau(t).x$).
\end{enumerate}

The following lemma describes some properties of the subset
$X^\zeta\subset X$.

\begin{Lem}\label{Lem:1.2}
\begin{enumerate}
\item If $X_0$ is an open saturated Stein subset of $X$, then $(X_0)^\zeta=X^\zeta\cap
X_0$.
\item The set $X^\zeta$ is open in $X$.
\item For any $y\in Y$ the intersection $X^\zeta\cap \pi^{-1}(Y)$
is a single $T$-orbit.
\item The map $\pi|_{X^\zeta}:X^\zeta\rightarrow Y$ is a locally trivial principal $T$-bundle whence
the quotient map for the
action $T:X^\zeta$.
\end{enumerate}
\end{Lem}
\begin{proof}
The limit in (2) exists in $X$ iff it exists in $\pi^{-1}(\pi(x))$
whence the first assertion. Using Theorem \ref{Thm:1.1}, we reduce
assertions 2,3 to the case when $X$ is a model Hamiltonian manifold.
It follows from assertion 3 that any fiber of $\pi|_{X^\zeta}$ is a
single orbit. Thus it is enough to prove assertion 4 also for model
manifolds only. So let $X$ be a model Hamiltonian $T$-manifold and
$T_0,\chi_1,\ldots,\chi_k,\lambda, \theta_i,\beta^i,v_j,v^j$ be such
as in Example \ref{Ex:1}. Write the decomposition
$\zeta=\sum_{i=1}^l a_i\theta_i+\sum_{i=1}^k b_i\chi_i$. The
condition that $\zeta$ is general implies that all $b_i$ are
nonzero. Thanks to Remark \ref{Rem:1.1}, we may assume that all
$b_i$ are positive. In  this case one checks directly that
$X^\zeta=\{x\in X| v_i(x)\neq 0, i=\overline{1,k}\}$ and assertions
2-4 follow.
\end{proof}

\begin{Lem}\label{Lem:1.1}
\begin{enumerate}
\item
$Y$ is smooth and $\psi$ is \'{e}tale.
\item
Let $T_0$ be a one-dimensional connected subgroup of $T$ and $Y_0\in
\D(T_0)$. Then  there is $\alpha\in \t^*$ such that $Y_0$ is a
connected component of $\psi^{-1}(\alpha+\t_0^\perp)$.
\item
For any $\chi\in \X(T)$ the sheaf $\Oo_\chi$ is a line bundle.
\item There is a 2-form $\omega_0\in \Omega^2(Y)$ and a $T$-invariant
one-form $\alpha\in \Omega^1(X)$
such that $\omega=\pi^*(\omega_0)-d\alpha$,
 and $\alpha(\xi_*)=H_\xi$ for any $\xi\in\t$.
Moreover, the class of $\omega_0$ in $H^2_{DR}(Y)$ does not depend
on the choice $\omega_0,\alpha$.
\item The 5-tuple $\Y_X$ is well-defined and Delzant.
\end{enumerate}
\end{Lem}
\begin{proof}[Proof of Lemma \ref{Lem:1.1}]
Everything but assertion 4 follows directly from Theorem
\ref{Thm:1.1} and Example \ref{Ex:1}. Let us prove assertion 4.

Let us check that $[\omega_0]\in H^2_{DR}(Y)$ is unique if
$\omega_0$ exists. Let
$\omega=\pi^*(\omega^1_0)-d\alpha^1=\pi^*(\omega^2_0)-d\alpha^2$ be
two representations of the required form. Then
$\alpha^0:=\alpha^1-\alpha^2$ is a $T$-invariant form such that
$\alpha^0(\xi_*)=0$. Choose a general for $X$ element $\zeta\in
\X(T)\otimes_\Z\R$. By assertion 4 of Lemma \ref{Lem:1.2}, the map
$\pi:X^\zeta\rightarrow Y$ is a principal $T$-bundle. It follows
that there is $\alpha_0\in \Omega^1(Y)$ such that
$\alpha^0=\pi^*(\alpha_0)$ on $X^\zeta$ and thence on the whole
manifold $X$. So $\pi^*(\omega_0^1-d\alpha_0)=\pi^*(\omega_0^2)$
whence $\omega_0^1-d\alpha_0=\omega_0^2$.

So it remains to prove that there exist $\omega_0,\alpha$ with the
required properties.

It follows from Example \ref{Ex:1} and Corollary \ref{Cor:1.1}
that there exists an open covering $Y=\cup_{i\in I}Y_i$ satisfying
the following conditions:
\begin{enumerate}
\item Both $Y_i$ and $\pi^{-1}(Y_i)$ are Stein.
\item $\omega|_{\pi^{-1}(Y_i)}=-d\alpha_i$, where $\alpha_i$ is a
holomorphic $T$-invariant 1-form on $\pi^{-1}(Y_i)$ such that
$\alpha_i(\xi_*)=H_\xi$.
\end{enumerate}

For a finite subset $J\subset I$ set $Y_J:=\cap_{j\in J}Y_j$.

Set $\alpha_{ij}:=\alpha_i-\alpha_j$. This is a $T$-invariant
holomorphic 1-form on $\pi^{-1}(Y_{ij})$ with
$\alpha_{ij}(\xi_*)=0$. It follows that there is a (unique)
holomorphic 1-form $\gamma_{ij}$ on $Y_{ij}$ such that
$\alpha_{ij}=\pi^*(\gamma_{ij})$. The collection $(\gamma_{ij})$
is a 1-cocycle in $\Omega^1$. Since $Y$ is Stein, we see that
there are holomorphic 1-forms $\gamma_i$ on $Y_i$ such that
$\gamma_{ij}=\gamma_i-\gamma_j$. Further,
$0=d\alpha_{ij}=\pi^*(d\gamma_{ij})$, whence $d\gamma_i=d\gamma_j$
on $Y_{ij}$. Set $\omega_0:=d\gamma_i$ on $Y_i$. Then
$\omega-\omega_0$ coincides with $d(\alpha_i-\pi^*(\gamma_i))$ on
$\pi^{-1}(Y_i)$. Since $\alpha_i-\pi^*(\gamma_i)=\alpha_j-\pi^*(\gamma_j)$
on $\pi^{-1}Y_j$, we may set $\alpha=\alpha_i-\pi^*(\gamma_i)$.
%Let us show, at first, that the map $H^2_{DR}(Y)\rightarrow
%H^2_{DR}(X)$ is injective. Indeed, let $\alpha\in \Omega^2(Y)$ be
%such that there is a $\beta\in \Omega^{1}(X)$ with
%$\pi^*\alpha=d\beta$. Replacing $X$ with $X^\zeta$ for general
%$\zeta$ we may assume that the action $T:X$ is free. Averaging
%$\beta$ with respect to $T$ (or, more precisely, with respect to the
%compact part of $T$), we may assume that $\beta$ is $T$-invariant.
%It follows that $d(\iota_{\xi_*}\beta)=0$ for any $\xi\in\t$, where
%$\iota_{\bullet}$ denotes the inner product of a vector field and a
%form. So $\iota_{\xi_*}\beta$ is a constant function.
%Suppose that $Y$ is an open disc. Then $X\cong Y\times T$. Choose
%coordinates $z_1,\ldots,z_n$ on $T\cong (\C^\times)^n$. The
%condition $\iota_{\xi_*}\beta=\const$ implies $\beta=\pi^*\beta_1+$
\end{proof}

\begin{Rem}\label{Rem:1.2}
Let $\Y_X=(Y,\psi,\D,(c_i)_{i=1}^n,c_0)$. Let us show how to recover
the structure of $X$ in a neighborhood of a point $x\in X$ with
closed orbit. Let $D_1,\ldots,D_k$ be all elements of $\D$
containing $\pi(x)$. For each $i$ there is a unique affine
hyperplane $\Gamma_i\subset \t^*$ containing $\psi(D_i)$. Let
$\eta_i,i=\overline{1,k},$ be a nonzero element from $\t$ lying in
the annihilator of the (linear) hyperplane associated with
$\Gamma_i$. Then the Lie algebra of $T_0$ is spanned by
$\eta_1,\ldots,\eta_k$, the character $\chi_i,i=\overline{1,k}$ is a
primitive element of $\X(T_0)$ annihilating all $\eta_j$ with $j\neq
i$ (a character $\chi_i$ is determined uniquely up to changing  the
sign), and $\lambda=\psi(x)$.
\end{Rem}

Finally, we prove that certain $T$-manifolds are Stein.

\begin{Prop}\label{Prop:1.1}
Let $X$ be a complex $T$-manifold such that the action $T:X$ is
effective, $Y$ a Stein manifold, and $\pi:X\rightarrow Y$ a
holomorphic $T$-invariant map. Suppose that there is an open
covering $Y_i,i\in I,$ of $Y$ by Stein submanifolds such that
$\pi^{-1}(Y_i)$ is a MF Hamiltonian Stein $T$-manifold and
$\pi:\pi^{-1}(Y_i)\rightarrow Y_i$ is the quotient map.  Then $X$
is Stein and $\pi:X\rightarrow Y$ is the quotient map.
\end{Prop}
\begin{proof}
We need to check two conditions (see \cite{Onishchik}, Subsection
4.1):
\begin{enumerate}
\item For any different points $x_1,x_2\in X$ there is a holomorphic
function $f\in \Oo(X)$ with $f(x_1)\neq f(x_2)$.
\item For any discrete sequence $x_n,n=1,2,\ldots,$ of points of $X$ there is a
holomorphic function $f\in \Oo(X)$ such that the sequence $f(x_n)$
is unbounded.
\end{enumerate}
Since $Y$ is Stein, we may assume that $\pi(x_1)=\pi(x_2)=y$ in (1)
and that in (2) the sequence $\pi(x_n)$ converges to $y$. By
Cartan's Theorem A, for any $\chi\in \X(T)$ there is a section
$f_\chi\in H^0(Y,\Oo_\chi)\subset \Oo(X)$ such that $f_\chi(y)\neq
0$ (recall that the sheaf $O_\chi$ is locally free, see Lemma
\ref{Lem:1.1}). Let $i$ be such that $y\in Y_i$. In (2) we may
assume that $\pi(x_n)$ lies in $Y_i$ for all $n$. Replacing $Y_i$
with a smaller neighborhood, we obtain that $\pi^{-1}(Y_i)$ is a
saturated neighborhood of a model Hamiltonian $T$-manifold.

Let $x_1,x_2$ be such as in (1). Assume that
$f_\chi(x_1)=f_\chi(x_2)$ for any $\chi$. Since $f_\chi$ generates
the  $\Oo(Y'_i)$-module $H^0(Y'_i,\Oo_\chi)$, where $Y_i'$ is an
appropriate open neighborhood of $y$, we see that $f(x_1)=f(x_2)$
for any $T$-semiinvariant function on $\pi^{-1}(Y_i)$. However,
there is a coordinate system on a model manifold consisting of
$T$-semiinvariant functions. Contradiction.

Now let $x_n, n=1,2,\ldots,$ be such as in (2). Analogously to the
previous paragraph, we get that the sequence $f(x_n)$ is bounded for
any $T$-semiinvariant function $f$. As above, this is absurd.
\end{proof}

%\section{Main construction}
\section{The sheaf $\Aut$}\label{SECTION_Aut}
Let $T,X,\pi, \Y_X=(Y,\psi,\D,(c_i)_{i=1}^n, c_0)$ be such as above.
The goal of this section is to study the sheaf of groups $\Aut^X$ on
$Y$ defined as follows: the group $\Aut^X(U)$ consists of all
Hamiltonian morphisms of  $\pi^{-1}(U)$ preserving $\pi$. This sheaf
plays a crucial role in the subsequent development, compare with the
proof of the uniqueness part of the Delzant theorem in
\cite{Woodward}.

At first, we construct a certain morphism of sheafs
$\Oo_Y\rightarrow \Aut^X$.

To any function $f\in \Oo(U)$, where $U$ is an open subset of $X$,
we assign its skew-gradient $v(f)$  by
$$\omega_x(v(f),\eta)=\langle\eta,d_xf\rangle, x\in U, \eta\in T_xX. $$

\begin{Lem}\label{Lem:2.1}
Let $U$ be an open subset of $Y$ and $f\in \Oo(U)\cong
\Oo(\pi^{-1}(U))^T$. Then the map $t\mapsto \exp(2\pi i tv(f))$ is a
well-defined  homomorphism $\C\rightarrow \Aut^X(U)$ such that the
corresponding action $\C:\pi^{-1}(U)$ is holomorphic and its
velocity vector field coincides with $2\pi i v(f)$.
\end{Lem}
\begin{proof}
Define two constant sheafs $\t_Y,T_Y$ on $Y$ with fibers $\t,T$,
respectively. There is the natural epimorphism $\t_Y\rightarrow T_Y,
\xi\mapsto \exp(2\pi i\xi),$ and the natural action
$T_Y(U):\pi^{-1}(U), \varphi.x=\varphi(\pi(x))x$ by holomorphic
$T$-equivariant automorphisms preserving $\pi$.

At first, suppose that the action $T:\pi^{-1}(U)$ is free. Note
that the vector field $v(f)$ is $T$-invariant and tangent to all
$T$-orbits (the latter stems from $\omega(v(f),\xi_*)=0$ for any
$\xi\in\t$). So we may consider $v(f)$ as a section of  $\t_Y$.
Applying the exponential map to $v(f)$, we get the one-parameter
subgroup $\exp(2\pi i t v(f))$ of $T_Y(U)$. So we get the action
$\C:\pi^{-1}(U)$, whose velocity vector field is $2\pi i v(f)$.
Since $v(f)$ is a Hamiltonian vector field, this action preserves
$\omega$.  Finally,  the action preserves $\pi$ and thus also
$\mu$.

Consider the general case. Choose general $\zeta\in
\X(T)\otimes_\Z\R$. It follows directly from the definition of
$X^\zeta$ that $\pi^{-1}(U)\cap X^\zeta$ is stable with respect to
the action $T_Y(U):\pi^{-1}(U)$. The action $\C:X^\zeta\cap
\pi^{-1}(U)$ constructed above is factorized through a homomorphism
$\C\rightarrow T_Y(U)$. Since $T_Y(U)$ acts on $\pi^{-1}(U)$ by
holomorphic automorphisms, we see that the action
$\C:\pi^{-1}(U)\cap X^\zeta$ can be extended to the whole set
$\pi^{-1}(U)$.
\end{proof}

So we have the sheaf morphism $\Oo_Y\rightarrow \Aut$ given by
\begin{equation}\label{eq:2.1}
f\mapsto \exp(2\pi i v(f)).
\end{equation}

\begin{Lem}\label{Lem:2.2}
The morphism of sheafs (\ref{eq:2.1}) is surjective.
\end{Lem}
\begin{proof}
Again, choose an element $\zeta\in \X(T)\otimes_\Z\R$ general for
$X$ and an open subset $U\subset Y$.  We get the natural inclusion
$\Aut^X(U)\hookrightarrow \Aut^{X^\zeta}(U)$ induced by the
restriction to $X^\zeta\cap \pi^{-1}(U)$. So we may replace $X$ with
$X^\zeta$ and assume that the action $T:X$ is free. Since the
question is local we may assume $X=T^*(T)$ and $U$ is given by (in
the notation of Example \ref{Ex:1}) $|\beta^i|<1, i=\overline{1,n}$.
We need to prove that the map $\Oo(U)\rightarrow \Aut^X(U)$ is
surjective.

The group $\Aut^X(U)\hookrightarrow T_Y(U)$ consists of all maps
$\Phi:U\rightarrow T$ such that the map $(t,y)\mapsto (\Phi(y)t,y)$
preserves $\omega$. Fix a basis $\theta_1,\ldots,\theta_n$ of
$\X(T)$. There is a holomorphic map
$\varphi=(\varphi_1,\ldots,\varphi_n):U\rightarrow \t$ such that
$\Phi=\exp(2\pi i \varphi)$. The inclusion $\Phi\in \Aut^X(U)$ holds
whenever the vector field $\sum_{i=1}^n \varphi_i
\theta_i\frac{\partial}{\partial \theta_i}$ is symplectic. The
latter is equivalent to the system of equations
$$\frac{\partial\varphi_i}{\partial
\beta^j}=\frac{\partial\varphi_j}{\partial \beta^i},
i,j=\overline{1,n}.$$ By the Dolbeaux lemma, there is $f\in \Oo(U)$
such that $\varphi_i=\frac{\partial f}{\partial \beta^i}$. It
follows that $v(f)=\varphi$.
\end{proof}

\begin{Lem}\label{Lem:2.3}
The kernel of the epimorphism (\ref{eq:2.1}) is $\C\oplus \X(T)^*$
(where $\C$ denotes the sheaf of constant functions and $\X(T)^*$
the sheaf of $H_\xi, \xi\in \X(T)^*$).
\end{Lem}
\begin{proof}
Again, it is enough to check this lemma for $X=T^*(T)$, where it is
checked directly.
\end{proof}

Finally, $\Aut^X\cong \Oo_Y/(\C\oplus \X(T)^*)$. In particular,
$\Aut^X$ depends only on $Y$ and $\psi:Y\rightarrow \t^*$, so we
write $\Aut$ instead of $\Aut^X$.

\begin{Cor}\label{Cor:2.4}
$H^j(Y,\Aut)=H^{j+1}(Y,\C\oplus\X(T)^*)$.
\end{Cor}
\begin{proof}
Since $Y$ is Stein, we have $H^j(Y,\Oo_Y)=0$ for $j>0$. It remains
to consider the long exact sequence associated with $0\rightarrow
\C\oplus \X(T)^*\rightarrow \Oo_Y\rightarrow \Aut\rightarrow 0$.
\end{proof}

\section{Uniqueness}\label{SECTION_uniqueness}
The goal of this section is to prove Theorem \ref{Thm:1}. Throughout
this section $X$ is a MF Hamiltonian Stein $T$-manifold and
$\Y_X=(Y,\psi, \D, (c_i)_{i=1}^n, c_0)$ the corresponding 5-tuple.

Denote by $\XX$ the set of all isomorphism classes of multiplicity
free Hamiltonian Stein $T$-manifolds with 5-tuples of the form
$(Y,\psi,\D,\bullet,\bullet)$. Two such Hamiltonian $T$-manifolds
$X,X'$ are supposed to be isomorphic if there is a Hamiltonian
isomorphism $\iota:X\rightarrow X'$  such that $\pi'\circ\iota=\pi$,
where $\pi,\pi'$ are the quotient maps for $X,X'$.

There is a natural  action of $H^1(Y,\Aut)$ on $\XX$, which we
describe now.

Choose $c\in H^1(Y,\Aut)$. Let $Y_i,i\in I,$ be an open covering and
$\varphi_{ij}\in \Aut(Y_{ij}),$ where $ Y_{ij}:=Y_i\cap Y_j$, be a
1-cocycle representing $c$. Set
$$X':=\coprod_{i\in I} \pi^{-1}(Y_i)/\sim,$$ where points $x_i\in
\pi^{-1}(Y_i), x_j\in \pi^{-1}(Y_j)$ are equivalent if they both
lie in $\pi^{-1}(Y_{ij})$  and $x_i=\varphi_{ij}(x_j)$. Since
$(\varphi_{ij})$ is a 1-cocycle, we see that $\sim$ is a genuine
equivalence relation. Clearly, $X'$ has a unique structure of a
Hamiltonian $T$-manifold such that the embedding
$\pi^{-1}(Y_i)\rightarrow X$ is a Hamiltonian morphism. It follows
from Proposition \ref{Prop:1.1} that $X'$ is Stein.

It is clear that if $\varphi'_{ij}$ is another 1-cycle and $X''$
is obtained by applying $\varphi'_{ij}$ to $X'$, then $X''$ is
obtained from $X$ by applying $\varphi'_{ij}\varphi_{ij}$. Now
suppose $\varphi_{ij}$ is a 1-coboundary, that is, there is a
1-cochain $f_i\in \Aut(Y_i)$ with $\varphi_{ij}=f_if_j^{-1}$. Then
there is the isomorphism $X\rightarrow X'$ given by $f_i$ on
$Y_i$. So $X'$ depends up to isomorphism only on $c$ and we write
$X_c$ for $X'$. Also we note that $X_{c_1c_2}=(X_{c_1})_{c_2}$, so
we do have an action of $H^1(Y,\Aut)$ on $\XX$.

The following proposition is the main property of this action.

\begin{Prop}\label{Prop:3.3}
Suppose $\XX$ is nonempty. Then the action of $H^1(Y,\Aut)$ on $\XX$
is free and transitive. Further, for any $(c'_i)_{i=0}^n\in
H^2(\C\oplus \X(T)^*)$ there is a unique element $X'\in \XX$ with
$\Y_X=(Y,\psi, \D,(c'_i)_{i=1}^n,c'_0)$.
\end{Prop}

The proof of this proposition will be given at the end of the
section after two auxiliary lemmas. Now we are going to derive
Theorem \ref{Thm:1} from Proposition \ref{Prop:3.3}.

\begin{proof}[Proof of Theorem \ref{Thm:1}]
Let $X'$ be a MF Hamiltonian Stein $T$-manifold and
$\varphi:\Y=(Y,\psi,\D,$ $(c_i)_{i=1}^n,c_0)\rightarrow
\Y_{X'}=(Y',\psi',\D',(c'_i)_{i=1}^n,c_0')$ be an arbitrary
morphism of $5$-tuples. Set $X=Y\times_{Y'}X'$ and let
$\widetilde{\psi}$ denote the map $X\rightarrow X'$ arising
from the Cartesian square. There is a unique structure of a
Hamiltonian $T$-manifold on $X$ such that $\widetilde{\psi}$ is a
Hamiltonian morphism. Being a closed submanifold in $Y\times X'$,
the manifold $X$ is Stein. The natural map $\pi:X\rightarrow Y$ is
the quotient map. It is clear from construction that $\Y_X=\Y$.

The previous construction reduces the proof of the theorem to the
case $\varphi=id$. This case stems from Proposition \ref{Prop:3.3}.
\end{proof}

Let $X'$ be as above,  $\pi'$ denote the quotient map $X'\rightarrow
Y$, and $\Y_{X'}=(Y,\psi',\D',(c'_{i})_{i=1}^n, c'_0)$. It is clear
from the construction of $X'$ that $\psi'=\psi, \D'=\D$. Let us
examine the behavior of $c_i'$. We recall that $H^1(Y,\Aut)\cong
H^2(X,\C\oplus \X(T)^*)$ (Corollary \ref{Cor:2.4}).

Recall also that there is a natural embedding of sheafs
$\Aut\hookrightarrow T_Y$, see the proof of Lemma \ref{Lem:2.1}.
So one can consider the homomorphism $\varphi\mapsto
\langle\chi,\varphi\rangle$ from $\Aut$ to the constant sheaf with
fiber $\C^\times$.

\begin{Lem}\label{Lem:3.1}
Choose $\chi\in \X(T)$ and let $c_\chi,c_\chi'$ be the classes of
the line bundles $\Oo_\chi, \Oo'_\chi$ corresponding to $X,X'$ in
$H^2(Y,\C)$. Let $c\in H^1(Y,\Aut)\cong H^2(Y,\X(T)^*\oplus \C)$.
Then $c'_\chi=c_\chi+\langle\chi,c\rangle$.
\end{Lem}
\begin{proof}
 Let $f_{ij}$ denote the transition functions for the
sheafs $\Oo_\chi$ on $Y$. From the construction of the isomorphism
$H^1(Y,\Aut)\cong H^2(Y,\C\oplus\X(T)^*)$ it follows that we need to
check that $f_{ij}\langle \chi,\varphi_{ij}\rangle$ is a system of
transition functions for $\Oo'_\chi$.

Indeed, let $\Oo'$ be the line bundle on $Y$ with transition
functions $f_{ij}\langle\chi,\varphi_{ij}\rangle$. Let $T$ act on
sections of $\Oo'$ by $\chi$. We may assume that $\Oo'$ is
trivialized over each $Y_i$. Let $\sigma$ be a global section of
$\Oo'$. Our claim will follow if we check that $\sigma$ can be
regarded as a function on $X'$ of weight $\chi$. To verify this we
need to check that $\sigma_i(x)=\sigma_j(\varphi_{ij}^{-1}x)$, where
$\sigma_i$ is the trivialization of $\sigma$. Here the subsets
$\pi^{-1}(Y_{ij})\subset \pi^{-1}(Y_i),\pi^{-1}(Y_j)$ are assumed to
be identified as in $X$. But
$\sigma_j(\varphi_{ij}^{-1}x)=\langle\chi,\varphi_{ij}\rangle^{-1}\sigma_j(x)=\sigma_i(x)$
(the factor $f_{ij}$ does not appear because of the choice of the
identification of two different $\pi^{-1}(Y_{ij})$).
%We may assume that $\Oo_\chi,\Oo'_\chi$ are trivial on each $Y_i$,
%that is there are  invertible functions $f_i\in \Oo(\pi^{-1}(Y_i)),
%f_i'\in \Oo(\pi'^{-1}(Y_i))$ of weight $\chi$.
\end{proof}

\begin{Lem}\label{Lem:3.2}
Let $c\in H^2(Y,\C)\hookrightarrow H^2(Y,\C\oplus\X(T)^*)\cong
H^1(Y,\Aut)$. Then $2\pi i c=c_0-c_0'$.
\end{Lem}
\begin{proof}
As in the proof of Lemma \ref{Lem:1.1}, we may assume that the
actions $T:X,T:X'$ are free. There is an $\Oo(Y)$-valued 2-cocycle
$(f_{ij})$ such that $\varphi_{ij}=\exp(2\pi i v(f_{ij}))$. Let
$\alpha_i$ be such as in the proof of Lemma \ref{Lem:1.1}. The
1-cocycle in closed 1-forms constructed for $\omega'$ as in the
proof of Lemma \ref{Lem:1.1} equals
$\alpha_j-\varphi_{ij}^{*-1}(\alpha_i)=(\alpha_j-\alpha_i)+
\alpha_i-\varphi_{ij}^{*-1}(\alpha_i)$. Let us compute
$\alpha_i-\varphi_{ij}^{*-1}\alpha_i$. We may assume that
$\pi^{-1}(Y_i)$ is an open neighborhood in $T^*(T)$. Let $\theta_i,
\beta^i$ be such as in Example \ref{Ex:1}. Then
$\alpha_i=\sum_{i=1}^n \beta^i d(\ln \theta_i)$. The map
$\varphi_{ij}$ equals $(\theta_i,\beta^i)\mapsto (\theta_i\exp(2\pi
i\frac{\partial f_{ij}}{\partial \beta^i}), \beta^i)$. Thus
$$\alpha_i-\varphi_{ij}^{*-1} \alpha_i=2\pi i\sum_{k=1}^n \beta^k
d(\frac{\partial f_{ij}}{\partial \beta^k})=2\pi i d(\sum_{k=1}^n
\beta^k\frac{\partial f_{ij}}{\partial \beta^k}-f_{ij}).$$ So the
1-cocycle in $H^1(Y,\Oo_Y/\C)$ corresponding to $c_0'-c_0$ is given
by $2\pi i(\mathcal{L}f_{ij}-f_{ij})$, where
$\mathcal{L}=\sum_{k=1}^n H_i\frac{\partial}{\partial H_i}$ (by
abusing notation, $\frac{\partial}{\partial H_i}$ stands for the
lifting of the corresponding vector field on $\t^*$ to $Y$). Note
that $\mathcal{L}(f_{ij}+f_{jk}+f_{ki})=0$, for
$f_{ij}+f_{jk}+f_{ki}=const$. Therefore the classes of $c_0-c_0'$
and $2\pi i c$ in $H^2(Y,\C)$ coincide.
\end{proof}

\begin{proof}[Proof of Proposition \ref{Prop:3.3}]
Let us prove that the action of $H^1(Y,\Aut)$ on $\XX$ is
transitive. Let $X,X'\in \XX$ and $\pi,\pi'$ be the corresponding
quotient maps. By Remark \ref{Rem:1.2}, there is an open covering
$Y_i,i\in I,$ of $Y$ such that there are Hamiltonian isomorphisms
$\iota_i:\pi^{-1}(Y_i)\rightarrow\pi'^{-1}(Y_i),i\in I,$ with
$\pi'\circ\iota_i=\pi$. So $\iota_i\circ\iota_j^{-1}$ is a 1-cocycle
in $\Aut$ and $X'=X_c$ for the corresponding cohomology class $c$.
It follows that the action is transitive.

Choose $X\in \XX$ and $(c_i')_{i=0}^n\in H^2(Y,\C\oplus\X(T)^*)$.
Let $\Y_X=(Y,\psi,\D, (c_i)_{i=1}^n, c_0)$. By Lemma \ref{Lem:3.1},
there is $c^1\in H^2(Y,\C\oplus \X(T)^*)$ such that
$\Y_{X_{c^1}}=(Y,\psi,\D,(c_i')_{i=1}^n, c_0'')$ for some $c_0'$ and
any two elements $c^1$ with this property differ by an element of
$H^2(Y,\C)$. Now, by Lemma \ref{Lem:3.2}, there is a unique element
$c^2\in H^2(Y,\C)$ such that $\Y_{X_{c^1c^2}}=(Y,\psi,\D,
(c_i')_{i=1}^n,c_0')$. This completes the proof.
\end{proof}

%\begin{Lem}\label{Lem:3.1}
%$X'$ is Stein.
%\end{Lem}
%\begin{proof}
%Replacing $Y_i$ with a finer covering, we may assume that both
%$Y_i,\pi'^{-1}(Y_i)$ are Stein, see Corollary \ref{Cor:1.1}. So
%\pi':X'\rightarrow Y$ is a Stein map in sense of \cite{Onishchik},
%Subsection 4.7. By \cite{Onishchik}, Theorem 4.22, $X'$ is Stein iff
%$H^1(X', \Oo_{X'})=0$. Choose a 1-cocycle $f_{ij}\in
%\Oo(\pi'^{-1}(Y_{ij}))$.
%\end{proof}

\section{Existence}\label{SECTION_existence}
The goal of this section is to prove Theorem \ref{Thm:2}. Thanks to
Proposition \ref{Prop:3.3}, it is enough to check that for any
5-tuple $\Y$ of the form $(Y,\psi,\D,(c_i)_{i=1}^n,c_0)$ with {\it
some} $c_i,i=\overline{0,n},$ there is $X$ with $\Y_X=\Y$.

Choose an open covering $Y=\Y_i$ such that $Y_i$ is an open disk
and there is an open saturated subset $X_i$ of some model
Hamiltonian $T$-manifold such that
$\Y_{X_i}=(Y_i,\psi|_{Y_i},\D_i,0,0)$, where $\D_i$ is the set of
components of $D\cap Y_i, D\in \D$. Let $\pi_i:X_i\rightarrow Y_i$
be the quotient map. Choose some isomorphisms
$\iota_{ij}:\pi_j^{-1}(Y_{ij})\rightarrow \pi_i^{-1}(Y_{ij})$ of
Hamiltonian manifolds such that $\pi_i\circ\iota_{ij}=\pi_j$. Such
isomorphisms exist by Theorem \ref{Thm:1.1}. We may assume that
$\iota_{ii}=id, \iota_{ji}=\iota_{ij}^{-1}$. Note, however, that, in
general, $\iota_{ij}\iota_{jk}\iota_{ki}\neq id$. So our task is
to modify these isomorphisms for the cocycle condition to hold.

There is $\zeta\in \X(T)\otimes\R$ general for all $X_i$. Using
Example \ref{Ex:1}, one easily gets that for any $i$ there is a
subvariety $\widetilde{Y}_i\subset X_i^\zeta$ satisfying the
following conditions:
\begin{enumerate}
\item The restriction of $\pi_i$ to $\widetilde{Y}_i$ is an
isomorphism $\widetilde{Y}_i\rightarrow Y_i$.
\item $\widetilde{Y}_i$ is a lagrangian submanifold of $X_i$.
\end{enumerate}

\begin{Prop}\label{Prop:4.1}
For any $i,j$ there is a unique element $\varphi_{ij}\in
\Aut(Y_{ij})$ such that
\begin{equation}\label{eq:4.1}\varphi_{ij}(\widetilde{Y}_i\cap
\pi_i^{-1}(Y_{ij}))=\iota_{ij}(\widetilde{Y}_j\cap
\pi_j^{-1}(Y_{ij})).\end{equation}
\end{Prop}
\begin{proof}
By Lemma \ref{Lem:1.2},
$\iota_{ij}(\pi_j^{-1}(Y_{ij})^\zeta)=\pi^{-1}_i(Y_{ij})^\zeta$.
Since the group $\Aut(Y_{ij})$ depends only on $Y_{ij},\psi$ and
does not depend on $\D$, we may assume that the actions $T:X_i,X_j$
are free. By (1), there is $\varphi_{ij}\in T_Y(Y_{ij})$ satisfying
(\ref{eq:4.1}). Note that $\varphi_{ij*}\xi_*=\xi_*,
\varphi_{ij}^*H_\xi=H_{\xi}$ for any $\xi$. Therefore for any $x\in
\pi_i^{-1}(Y_{ij}), \eta\in T_x(\pi^{-1}_i(Y_{ij}))$ we have
$$\omega_{\varphi_{ij}(x)}(\xi_*,\varphi_{ij*}\eta)=\partial_{\varphi_{ij*}\eta}H_\xi(\varphi_{ij}x)=
\partial_\eta H_\xi(x)=\omega_x(\xi_*,\eta).$$
 Since both
$\widetilde{Y}_i\cap
\pi_i^{-1}(Y_{ij}),\iota_{ij}(\widetilde{Y}_j\cap
\pi_j^{-1}(Y_{ij}))$ are lagrangian, we easily see that
$\varphi_{ij}$ is a symplectomorphism. Thus $\varphi_{ij}\in
\Aut(Y_{ij})$.
\end{proof}

Set $\widetilde{\iota}_{ij}:=\varphi_{ij}^{-1}\iota_{ij}$. By the
construction of $\varphi_{ij}$, $$\widetilde{Y}_i\cap
\pi_i^{-1}(Y_{ij})=\widetilde{\iota}_{ij}(\widetilde{Y}_j\cap
\pi_j^{-1}(Y_{ij})).$$ This equation and (1) imply the cocycle
condition
$\widetilde{\iota}_{ij}\widetilde{\iota}_{ji}\widetilde{\iota}_{ki}=id$.

Set $X:=\sqcup_{i} X_i/\sim$, where points $x_i\in X_i, x_j\in X_j$
are equivalent if $\pi_i(x_i)=\pi_j(x_j)\in Y_{ij}$ and
$x_i=\widetilde{\iota}_{ij}(x_j)$. The manifold $X$ is Stein, see
Proposition \ref{Prop:1.1}, it has the natural Hamiltonian
structure and the action $T:X$ is MF. By construction,
$\Y_X=(Y,\psi,\D,(c_i)_{i=1}^n,c_0)$ for some $c_i$.

\section{An open problem}\label{SECTION_open}
Here we would like to state an open problem  communicated to us by
F. Knop. Until a further notice $T$ is a compact torus. Recall that
any MF compact Hamiltonian $T$-manifold admits an invariant
K\"{a}hler structure. This follows, for example, from the Delzant
construction involving a Hamiltonian reduction.

For a moment, let $X$ be an arbitrary smooth manifold. Recall that a
{\it hyperk\"{a}hler} structure on $X$ is a quadruple $(q,I,J,K)$,
where $q$ is a Riemannian metric, and $I,J,K$ are complex structures
satisfying the quaternionic relations $IJ=K, JK=I, KI=J$ and such
that the three 2-forms defined by $\omega_A(u,v)=q(Au,v), A=I,J,K,$
are symplectic. The basic example here is $X=\R^{4n}$ considered as
an $n$-dimensional quaternionic vector space $\H^n$. Note that
$\omega_J+i\omega_K$ is a holomorphic symplectic form with respect
to the complex structure $I$.

Now let $X$ be a hyperk\"{a}hler manifold equipped with an action of
$T$ preserving the hyperk\"{a}hler structure. This action is called
{\it hyperhamiltonian} if it is Hamiltonian for all symplectic forms
$\omega_A, A=I,J,K$. Let $\mu_I,\mu_J,\mu_K$ be the corresponding
moment maps. Then one easily checks that $\mu_J+i\mu_K$ is a
holomorphic map with respect to  the complex structure $I$. If one
can lift the action $T:X$ to a holomorphic action $T_\C:X$, where
$T_\C$ denotes the complexification of $T$, then $X$ becomes a
Hamiltonian $T_\C$-manifold.

We say that a hyperhamiltonian action $T:X$ is MF if $\dim
X=4(\dim T-\dim T_0)$, where $T_0$ denotes the inefficiency kernel
for the action $T:X$.

Actually, there is a special class of MF hyperhamiltonian actions
studied extensively in the last ten years, so called {\it
hypertoric manifolds} (or toric hyperk\"{a}hler manifolds), see,
for example, \cite{BD},\cite{HS},\cite{Konno},\cite{Proudfoot}.
Let us give their definition.

There is a reduction procedure for hyperhamiltonian manifolds
introduced in a more general setup in \cite{HKLR}. Choose a
connected subgroup $T_0\subset T$. For $\alpha\in(\t_0^*)^{\oplus
3}, \alpha=(\alpha_1,\alpha_2,\alpha_3)$ we put $X\dquo_\alpha
T:=(\mu_0^{-1}(\alpha))/T$, where
$\mu=(p\circ\mu_I,p\circ\mu_J,p\circ\mu_K), p:\t^*\rightarrow
\t_0^*$ is the natural projection. If $T$ acts freely on
$\mu^{-1}(\alpha)$, then $X\dquo_\alpha T$ is a genuine manifold
of dimension $\dim X-4\dim T$ possessing a natural hyperk\"{a}hler
structure. Moreover, $X$ is a hyperhamiltonian $T/T_0$-manifold.
By a hypertoric manifold one means a reduction of $\H^n$ by a
subtorus in a maximal torus of $\operatorname{Sp}(n)$.

\begin{Prob}
Let $X$ be a MF  Hamiltonian Stein $T_\C$-manifold with symplectic
form $\omega$ and moment map $\mu$. Does there exist a
hyperk\"{a}hler structure $(q,I,J,K)$ on $X$ such that
\begin{enumerate}
\item $I$ coincides with the initial complex structure on $X$.
\item $\omega_J=\Re\omega, \omega_K=\Im\omega$.
\item The action $T:X$ is hyperhamiltonian with $\mu_J=\Re\mu,
\mu_K=\Im\mu$.
\end{enumerate}
\end{Prob}

{\Small Chair of Higher Algebra, Department of Mechanics and
Mathematics, Moscow State University.

\noindent
E-mail address: ivanlosev@yandex.ru}
\end{document}